\documentclass[final]{siamltex}

\usepackage{amsmath}
\usepackage{amssymb}
\PassOptionsToPackage{normalem}{ulem}
\newcommand{\menge}[2]{\big\{{#1} \mid {#2}\big\}}

\newcommand{\emp}{\ensuremath{{\varnothing}}}

\newcommand{\scal}[2]{\left\langle{#1}\mid {#2} \right\rangle}

\newcommand{\HH}{\ensuremath{\mathcal H}}
\newcommand{\SG}{\ensuremath{{\mathfrak B}}}

\newcommand{\GG}{\ensuremath{\mathcal G}}

\newcommand{\E}{\ensuremath{\mathbb{E}}}

\newcommand{\RR}{\ensuremath{\mathbb R}}

\newcommand{\RP}{\ensuremath{\left[0,+\infty\right[}}
\newcommand{\RPP}{\ensuremath{\,\left]0,+\infty\right[}}

\newcommand{\NN}{\ensuremath{\mathbb N}}

\newcommand{\dom}{\ensuremath{\operatorname{dom}}}

\newcommand{\prox}{\ensuremath{\operatorname{prox}}}

\newcommand{\argmin}{\ensuremath{\operatorname{argmin}}}

\newcommand{\zer}{\ensuremath{\operatorname{zer}}}
\newcommand{\gra}{\ensuremath{\operatorname{gra}}}

\newcommand{\PP}{\ensuremath{\boldsymbol{P}}}

\newcommand{\Id}{\ensuremath{\operatorname{I}}}

\newcommand{\weakly}{\ensuremath{\rightharpoonup}}

\newcommand{\pinf}{\ensuremath{+\infty}}

\newtheorem{assumption}[theorem]{Assumption}
\newtheorem{algorithm}[theorem]{Algorithm}
\newtheorem{example}[theorem]{Example}
\newtheorem{problem}[theorem]{Problem}
\newtheorem{remark}[theorem]{Remark}
\RequirePackage[colorlinks,hyperindex]{hyperref}
\numberwithin{equation}{section}
\renewcommand\theenumi{(\roman{enumi})}
\renewcommand\theenumii{(\alph{enumii})}

\newcommand{\obj}{\Phi}
\newcommand{\smo}{L}
\newcommand{\nsm}{G}

\title{Stochastic forward-backward splitting for monotone inclusions
in Hilbert spaces}
 \author{Lorenzo Rosasco  $^\dag$\thanks{
       DIBRIS, Universit\`a di Genova,
       Via Dodecaneso, 35,
       16146, Genova, Italy, ({\tt lrosasco@mit.edu})}   
 \and
Silvia Villa $^\dag$ \and B$\grave{\text{\u{a}}}$ng C\^ong V\~u  
\thanks{
 LCSL, Istituto Italiano di Tecnologia
       and Massachusetts Institute of Technology,
       Bldg. 46-5155, 77 Massachusetts Avenue, Cambridge, MA 02139, USA, ({\tt Silvia.Villa@iit.it, Cong.Bang@iit.it})} 
}

\begin{document}
\maketitle
\frenchspacing

\begin{abstract}
We propose and analyze the convergence of a novel stochastic forward-backward splitting algorithm 
 for solving monotone inclusions given by the sum of a maximal monotone
operator and a single-valued maximal monotone cocoercive operator. This latter framework
has a number of interesting special cases, including variational inequalities and convex minimization
problems,   while  stochastic approaches are practically relevant to account for perturbations in the data.     
The algorithm we propose  is a stochastic extension of  the classical deterministic forward-backward method, and is obtained  considering  the composition of the resolvent of the maximal monotone operator with  
a forward step based on a stochastic estimate of the single-valued operator. 
Our study  provides a non-asymptotic error analysis  in expectation for the strongly 
monotone case, as well as  almost sure convergence   under weaker assumptions. 
The approach we consider allows to avoid averaging, a feature critical when considering methods based on sparsity, and,  for  minimization problems, it allows  to obtain convergence rates  matching  those obtained by stochastic extensions of so called accelerated methods.
Stochastic quasi Fejer's sequences are  a key technical  tool  to prove almost sure convergence. 
\end{abstract}


\title{Stochastic forward-backward splitting for monotone inclusions
}


\maketitle

\section{Introduction}
Maximal monotone operators have been extensively studied
since \cite{Min62}, largely because they have wide applications in pure and applied sciences \cite{Bre73,PasSbu78},  
and because they provide a convenient framework for a unified treatment of equilibrium problems, 
variational inequalities, and convex optimization, see \cite{livre1,Bre73,Zei90} and references therein. 
More precisely, let $\HH$ be a real Hilbert space and let $T\colon \HH\to 2^{\HH}$ be  a set 
valued maximal monotone operator \cite{livre1},  a key problem is to find an element
$\overline{w}\in\HH$  such that $0\in T(\overline{w})$.  In this paper, we consider the case where $T$ is the sum of two maximal 
monotone operators  denoted by $A$ and $B$, with $B$ single valued and cocoercive. The problem  is
then  to find  $\overline{w} \in\HH$ such that 
\begin{equation}
\label{e:inc}
 0\in A\overline{w} + B\overline{w},
\end{equation}  
under the assumption that such a $\overline{w}$ exists.
Problem   \eqref{e:inc} includes for example  fixed point
problems and variational
inequalities, which can be recovered choosing $A$ equal to the normal cone of a nonempty closed and convex set, 
but also composite minimization problems, where $A$ is the subdifferential of a proper 
lower semicontinuous convex functions and $B$ is the gradient of a smooth convex function \cite{Roc70,Roc70a}.
Indeed, there is a vast literature on  algorithmic schemes for  solving \eqref{e:inc}, and in particular on approaches 
separating the contribution of $A$ and $B$.  Well-known among such approaches is the forward-backward splitting algorithm 
\cite{livre1,Com04},
\begin{equation}\label{eq:bfba}
(\forall n\in\mathbb{N}) \quad w_{n+1}= J_{\gamma_n A}(w_n-\gamma_n B w_n),
\end{equation}  
where  $\gamma_n\in\left]0,+\infty\right[$ and $J_{\gamma_n A}=(I+\gamma_n A)^{-1}$ is the resolvent operator of $A$.  
Since the seminal works in \cite{LioMer79,Pas79}, forward-backward splitting methods have been considerably developed to be more flexible, 
to achieve better convergence properties, and to allow for numerical errors, 
see \cite{livre1,beck09,ComVu13,siam05,nesterov07,VilSal13} and references therein. 

In this paper, we are interested in the practically relevant situation in which
 $B$ is known only through measurements subject to non vanishing random noise, or
 when the  computation of a stochastic estimate of $B$ is
 cheaper than the evaluation of the operator itself.   While there is a rich literature on stochastic proximal gradient splitting algorithms for convex minimization problems 
(see e.g. \cite{Duchi09,AtcForMou14}), and various results 
for variational inequalities are available \cite{Jud11,JiaXu08,KosNedSha13}, we are not aware of previous study of stochastic splitting algorithms  for solving monotone inclusions.
The results in this paper fill this gap  by proposing a stochastic forward backward splitting method 
 for monotone inclusions and proving: 1)  a non-asymptotic  error analysis  in expectation, and 2) strong almost sure convergence  of the iterates. More specifically, under strong monotonicity assumptions,  we provide non asymptotic bounds for convergence in 
norm and in expectation, leveraging on a non asymptotic version of Chung's lemma \cite[Chapter 2, Lemma 5]{Pol87}. 
Almost sure convergence is obtained under the weaker assumption of uniform monotonicity of $B$ using   the concept of stochastic quasi-Fej\`er sequences \cite{ermol1969method,Ermol68},  
which allows to originally combine deterministic and stochastic techniques. 

A few  features of our approach are worth mentioning. First,  from the modeling point of view, we
consider weak assumptions on the stochastic estimate of $B$ with respect to the ones 
usually required in the stochastic optimization literature, see e.g. \cite{Nem09}. 
In particular, the assumption on the stochastic estimate of $B$ is different from the
one in \cite{ComPes14}, which assume a summability condition on the errors  (on the other hand, 
the fact that in our case the errors do not go to zero does not allow to consider a non vanishing step-size, 
as it is done in \cite{ComPes14}).
Second, our approach allows to  avoid averaging  the iterates. This aspect becomes crucial in situations where 
the problem structure  is meant to enforce  sparsity of the solution and an averaging process can be 
detrimental, see e.g.  \cite{LinChePen14}. 

 
The paper is organized as follows. 
Section \ref{sec:pre} collects some preliminaries proved in the Appendix. In Section \ref{sec:main}
we establish the main results of the paper: almost sure convergence of the iterates and 
a non-asymptotic analysis of stochastic forward-backward splitting algorithm,
where the bounds depend explicitly on the  parameters of the problem. 
Section \ref{sec:sc} focuses on special cases:
variational inequalities, minimization problems, and minimization problems over orthonormal bases. 
For the case of variational inequalities, 
we obtain an additional convergence result    
without imposing stronger monotonicity properties on $B$, which requires averaging of the 
iterates.

\noindent{\bf Notation.} Throughout, $(\boldsymbol{\Omega, \mathcal{A},\mathsf{P}})$ is a probability space,
$\NN^* = \NN\backslash\{0\}$,
 $\HH$  is a real Hilbert space, and   $2^{\HH}$ is its power set.
We denote by  $\scal{\cdot}{\cdot}$ and $\|\cdot\|$  the scalar product and the associated norm
of $\HH$.  
The symbols $\weakly$ and $\to$ denote weak and strong convergence, respectively. We denote  
by $\ell_+^1(\NN)$ the set of summable sequences in  $\RP$.
The class of all lower semicontinuous convex functions 
$G\colon\HH\to\left]-\infty,+\infty\right]$ such 
that $\dom G=\menge{x\in\HH}{G(x) < +\infty}\neq\emp$ 
is denoted by $\Gamma_0(\HH)$.
 We denote by 
$\sigma(X)$ the  $\sigma$-field generated by a random variable $X\colon \boldsymbol{\Omega}\to\HH$, 
where $\HH$ is endowed with the Borel $\sigma$-algebra. The 
expectation of a random variable $X$ is denoted by $\E[X]$. The conditional expectation
of $X$ given a sub sigma algebra $\mathcal{F}\subset \boldsymbol{\mathcal{A}}$ is denoted by 
$\E[X|\mathcal{F}]$. The conditional expectation of $X$ given $Y$ is denoted by $\E[X|Y]$.
A sequence $(\mathcal{F}_n)_{n\in\mathbb{N}}$ of sub sigma algebras of $\boldsymbol{\mathcal{A}}$
such that, for every $n\in\NN$, $\mathcal{F}_n\subseteq \mathcal{F}_{n+1}$ is called a filtration.
Let, for every  $n\in\mathbb{N}$,  $X_n\colon\boldsymbol{\Omega}\to\HH$ be an integrable  random variable
with $\E[\|X_n\|]<+\infty$. The sequence $(X_n)_{n\in\mathbb{N}}$ is called a random process.


\section{Preliminaries}\label{sec:pre}
Before discussing our main contributions, we recall  basic concepts and results we need in the following.

Let $A\colon\HH\to 2^{\HH}$ be a set-valued operator.
The domain and the graph of $A$ are respectively defined by 
$\dom A=\menge{w\in\HH}{Aw\neq\emp}$ and 
$\gra A=\menge{(w,u) \in \HH\times\HH}{u\in Aw}$.
We denote by $\zer A=\menge{w\in\HH}{0\in Aw}$ the set of zeros 
of $A$. The inverse of $A$ is
$A^{-1}\colon\HH\mapsto 2^{\HH}\colon u\mapsto 
\menge{w\in\HH}{u\in Aw}$. 
The following notion is central in the paper.
\begin{definition}  Let $A\colon\HH\to 2^{\HH}$ be a set-valued operator. $A$ is monotone if 
\begin{equation}\label{eq:mon}
\big(\forall (w,u)\in\gra A\big)\big(\forall (y,v)\in \gra A\big) \quad \scal{w-y}{u-v} \geq 0,
\end{equation}
and maximally monotone if it is monotone and there exists no 
monotone operator  $B\colon\HH\to2^\HH$ such that $\gra B$ properly contains $\gra A$. 
\end{definition}

Let $A$ be a monotone operator and let $y\in\dom A$. 
The following concepts  are pointwise 
variants of the well-known concepts of uniform and strong monotonicity. 

We say that $A$ is uniformly monotone 
at $y$ if there exists an 
increasing function $\phi\colon\left[0,+\infty\right[\to 
\left[0,+\infty\right]$ vanishing only at $0$ such that 
\begin{equation}\label{oioi}
\big(\forall (w,u)\in\gra A\big)\big(\forall v\in  Ay\big)
\quad\scal{w-y}{u-v}\geq\phi(\|w-y\|).
\end{equation}
In the case when $\phi = \mu|\cdot|^2$, for some $\mu\in\left]0,+\infty\right[$, we say that 
$A$ is $\mu$-strongly monotone at $y$. If $A-\mu\Id$ is  monotone, for some $\mu\in\left]0,+\infty\right[$, 
we say that $A$ is $\mu$-strongly monotone. 
We say that $A$ is strictly monotone at $y\in\dom A$ if 
\begin{equation}
\big(\forall (w,u)\in\gra A\big)\big(\forall v\in Ay\big)\quad w\neq y \Rightarrow
\quad\scal{w-y}{u-v} > 0.
\end{equation}
Let $\beta\in \left]0,+\infty\right[$. A single-valued operator 
$B\colon \HH\to\HH$ is $\beta$-cocoercive  if
\begin{equation}
 (\forall w\in\HH)(\forall y\in\HH)\quad 
\scal{w-y}{Bw-By} \geq \beta \|Bw-By\|^2.
\end{equation}
The resolvent of any maximally monotone operator $A$ is 
\begin{equation}
 J_A=(\Id + A)^{-1}. 
\end{equation}
 We recall that $J_A$  is well defined and single valued \cite{Min62}, and can therefore be identified
 with an operator $J_A\colon\HH\to\HH$.
When $A=\partial G$ for some $\nsm \in \Gamma_0(\HH)$, then $J_A$
coincides with the proximity operator of $\nsm$ \cite{Mor62}, which is defined as
\begin{equation}
\label{e:prox}
\prox_\nsm\colon\HH\to\HH\colon w
\mapsto\underset{v\in\HH}{\argmin}\: \nsm(v) + \frac12\|w-v\|^2.
\end{equation}

We next recall the concept of stochastic quasi  Fej\'er sequence, which was introduced 
and studied in the papers \cite{Erm71,ermol1969method,Ermol68}.
This concept provides a unified approach to prove convergence of several algorithms in 
convex optimization (see \cite{livre1} and references therein).

\begin{definition}{\rm\cite{Ermol68}}
 Let $S$ be a non-empty subset of $\HH$.
 A random process $(w_n)_{n\in\NN^*}$ in $\HH$ is stochastic
 quasi-Fej\'er monotone with respect to the set  $S$ if $\E[\|w_1\|^2] < +\infty$ and  there exists $(\varepsilon_n)_{n\in\NN^*}\in\ell_{+}^1(\NN^*)$ such that 
\begin{equation}\label{e:fejer}
 (\forall w\in S)(\forall n\in\NN^*)\quad \E[\|w_{n+1}-w\|^2|\sigma(w_1,\ldots, w_n)] \leq \|w_n-w\|^2 + \varepsilon_n \quad \text{a.s.}
\end{equation}
\end{definition}

\section{Main results} \label{sec:main}
The following is the main problem studied in the paper. 

\begin{problem}
\label{inc0}
Let $A\colon \HH\to 2^{\HH}$ be a maximally monotone operator, let $\beta\in\,]0,+\infty[$ and let $B\colon\HH\to\HH$ 
be a $\beta$-cocoercive operator. Assume that $ \zer (A+B)\neq \varnothing$.
The goal is to  find $\overline{w} \in\HH$ such that 
\begin{equation}
\label{e:inc0}
 0\in A\overline{w} + B\overline{w} .
\end{equation}  
\end{problem}
 
\subsection{Algorithm}

We propose the following stochastic forward-backward splitting algorithm 
for solving Problem \ref{inc0}. The key difference with respect to the classical setting is that we assume
to have access only to a stochastic estimate of $B$.  
\begin{algorithm} 
\label{a:maininc}
Let $(\gamma_n)_{n\in\NN^*}$ be a  sequence in $]0,+\infty[$,
let $(\lambda_n)_{n\in\NN^*}$ be a sequence in $\left[0,1\right]$,
and let $(\SG_n)_{n\in\NN^{*}}$ be a $\HH$-valued random process such that 
$(\forall n\in\NN^*)\; \E[\| \SG_n\|^2]  < +\infty$. 
Let $w_1\colon\Omega \to \HH$ be a random variable such that $\E[\|w_1\|^2]<+\infty$
and set
\begin{equation}
\label{e:main1*}
(\forall n\in \NN^{*})\quad
\begin{array}{l}
\left\lfloor
\begin{array}{l}
z_n = w_n- \gamma_n \SG_n\\
y_{n} = J_{\gamma_n A}z_n\\
w_{n+1} = (1-\lambda_n) w_{n} + \lambda_ny_{n}.\\
\end{array} 
\right.\\[2mm]
\end{array}
\end{equation}
\end{algorithm}
\vspace{0.8cm}

We will consider the following conditions for the filtration $(\mathcal{F}_n)_{n\in\NN^*}$, $\mathcal{F}_n= \sigma(w_1,\ldots,w_n)$.
\begin{enumerate}
\item[(A1)] 
 For every $n\in\NN^*$, $\E[\SG_n|\mathcal{F}_n] = Bw_n$.
\item[(A2)] 
There exist  $(\alpha_n)_{n\in\NN^*}$ in $\left]0,+\infty\right[$ and $\sigma\in\RPP$ such that,
 for every $n\in\NN^*$,  $\E[\|\SG_n - Bw_n\|^2|\mathcal{F}_n] 
\leq \sigma^2(1+ \alpha_n\|Bw_n\|^2 )$.
\item[(A3)] 
There exists $\varepsilon\in\,]0,+\infty[ $ such that $(\forall n\in\NN^*)$
$\gamma_n\leq (2-\epsilon)\beta/(1+2\sigma^2\alpha_n) $.
\item[(A4)] Let $\overline{w}$ be a solution of Problem \ref{inc0} and let 
\[
(\forall n\in\mathbb{N}^*)\quad \chi^{2}_n =
\lambda_n\gamma_{n}^2
\big(1 +2\alpha_{n}\|B\overline{w}\|^2\big).
\] Then
the following hold:
\begin{equation}
\label{e:suma}
 \sum_{n\in\NN^{*}}\lambda_n\gamma_n = +\infty
\quad \text{and}\quad 
 \sum_{n\in\NN^{*}}\chi^{2}_n < +\infty.
\end{equation}
\end{enumerate}
\begin{remark}\ 
\begin{enumerate}
\item
If, for every $n\in\NN^*$,  $\SG_n = Bw_n$, Algorithm \ref{a:maininc} reduces to the well known 
forward--backward splitting in \cite[Section 6]{Siop1}. 
However, under Assumptions (A1)-(A2)-(A3)-(A4),   weak convergence 
of $(w_n)_{n\in\NN^*}$ is not guaranteed since the conditions $\sum_{n\in\NN^*}\lambda_n\gamma_{n}^2<+\infty$ and 
$\sum_{n\in\NN^*} \lambda_n\gamma_n=+\infty$ imply $\inf\gamma_n=0$, while
to apply the classic theory we need $\inf \gamma_n$ to be strictly greater
than 0.  Assuming additionally that $(\forall n\in\mathbb{N}^*)\; \lambda_n=1$, under our assumptions
only ergodic convergence of $(w_n)_{n\in\mathbb{N}^*}$ has been proved in the deterministic case 
in \cite{Pas79}.
\item
A stochastic forward-backward splitting algorithm for monotone inclusions
has been recently analyzed in \cite{ComPes14}, under rather different assumptions. Indeed, 
they consider a fixed stepsize and a summability condition on $\E[\|\SG_n - Bw_n\|^2|\mathcal{F}_n]$.
In the case $A = \partial\nsm$ and $B = \nabla\smo$, for some $\nsm$ and $\smo \in \Gamma_0(\HH)$ such that $\smo$ is differentiable 
with $\beta^{-1}$-Lipschitz continuous gradient, Algorithm \ref{a:maininc} reduces to the stochastic
proximal forward-backward splitting which  is a variant of the algorithm in 
\cite{Duchi09}, also  studied in \cite{AtcForMou14}. 
\item Condition (A2) can be seen as a relative error criterion, and has been considered in \cite{Barty07} for the case of constrained minimization problems on infinite dimensional spaces. This is a more general condition than the one usually assumed in the context of stochastic optimization, where $\alpha_n=0$. 
\item
If $A=0$, then $B\overline{w}=0$ for every solution of Problem \ref{inc0}. In this case, $\chi_n=\lambda_n \gamma_n^2$ in \eqref{e:suma} and condition $(A4)$ becomes $\sum_{n\in\NN^*} \lambda_n\gamma_n=+\infty$ and $\sum_{n\in\NN^*} \lambda_n\gamma_n^2<+\infty$. The latter are the usual conditions required on the stepsize in the study of stochastic gradient descent algorithms  (see e.g. \cite{Bert00}). 
These conditions guarantee a sufficient, but not too fast, decrease of the stepsize length. Moreover, in this case,  $(\alpha_n)_{\in\NN^*}$ is not necessarily bounded, therefore condition (A2) is a very weak requirement. 
\end{enumerate}
\end{remark}
\begin{example}\label{ex:afm} Let $(\GG,\mathcal{B},P)$ be a probability space, let 
$b\colon\HH\times\GG\to \HH$ be a measurable function such that $\int_{\GG}\|b(w,y)\|P(dy)<+\infty$,
and suppose that $B$ satisfies
\begin{equation}
(\forall w\in\HH)\quad Bw = \int_{\GG}b(w,y)P(dy).
\end{equation}
Then there are several ways to find a stochastic estimate of $Bw$. In particular, if an independent
and identically distributed sequence $(y_n)_{n\in\NN^*}$ of realizations of the random vector $y$
 is available, then   one can take $\SG_n=b(w_n,y_n)$.
If in addition $B$ is a gradient operator and $\GG$ is finite dimensional, we are in the classical 
setting of stochastic optimization \cite{Nem09}.
\end{example}

\subsection{Almost sure convergence}
In this section we describe our main results about almost sure convergence of the iterates 
of Algorithm \ref{a:maininc}. All the proofs are postponed to Section \ref{sec:pf}.
We first collect basic properties satisfied by
 the sequences in Algorithm  \ref{a:maininc}.
\begin{proposition}
\label{p:1}
Suppose that (A1), (A2), (A3), and (A4) are satisfied.
Let $(w_n)_{n\in\NN^{*}}$ be the sequence generated by Algorithm \ref{a:maininc} and let $\overline{w}$ be a solution
of Problem \ref{inc0}.
Then the following hold:
\begin{enumerate}
\item \label{p:1i}
The sequence $(\E[\| w_n-\overline{w}\|^{2}])_{n\in\NN^{*}}$ converges to a finite value.
\item \label{p:1iii}
$\sum_{n\in\NN^{*}}\lambda_n\gamma_n \E[\scal{w_n-\overline{w}}{Bw_n - B\overline{w}}] < \pinf$.
Consequently, 
\begin{equation*}
\varliminf_{n\to\infty}\E[\scal{w_n-\overline{w}}{Bw_n-B\overline{w}}] = 0
\quad \text{and} \quad
 \varliminf_{n\to\infty}\E[\|Bw_n-B\overline{w}\|^2] = 0.  
\end{equation*}
\item
\label{p:1ii}  
$\displaystyle\sum_{n\in\NN^{*}}
\! \lambda_n\E[\|w_n\!-y_n-\gamma_n(\SG_n-B\overline{w})\|^2]\! <\! \pinf $
and 
$\displaystyle\sum_{n\in\NN^{*}} \!\lambda_n \E[\|w_n-y_n \|^2]\! <\! +\infty$.
\end{enumerate}
\end{proposition}

Proposition \ref{p:1} states similar properties to those stated for the forward-backward splitting algorithm
in \cite{siam05}. These properties are key to prove almost sure convergence, which is stated in the
next theorem. Depending on the monotonicity properties of the operator $B$, we get different 
convergence results. 

\begin{theorem} 
\label{t:2nv}  
Suppose that conditions (A1), (A2), (A3) and (A4) are satisfied.
Let $(w_n)_{n\in\NN^{*}}$ be the sequence generated by Algorithm \ref{a:maininc} and let $\overline{w}$ be a solution 
of Problem \ref{inc0}.
Then the following hold:
\begin{enumerate}
\item 
\label{t:2nvi-} $(w_n)_{n\in\NN^*}$ is stochastic quasi-Fej\`er monotone with respect to $\zer(A+B)$.
\item 
\label{t:2nvi} There exists an integrable random variable $\zeta_{\overline{w}} $ such that 
$\|w_n-\overline{w}\|^2\to\zeta_{\overline{w}}$ a.s.  
\item\label{t:2nvii} 
If $B$ is uniformly monotone at $\overline{w}$, 
then $w_n \to \overline{w}$ a.s.
\item \label{t:2nviii}
If $ B$  is strictly monotone at $\overline{w}$ and  weakly continuous, 
then there exists $\Omega_1\in\boldsymbol{\mathcal{A}}$ such that $\boldsymbol{\mathsf{P}}(\Omega_1)=1$, and,
for every $\omega\in\Omega_1$, there exists a subsequence $(w_{t_n}(\omega))_{n\in\NN^{*}}$ such that
 $w_{t_n}(\omega)\weakly\overline{w}$.
\end{enumerate}
\end{theorem}

This kind of convergence of the iterates is the one traditionally studied in the stochastic optimization
literature.  However, most papers focus on the finite dimensional setting, and require boundedness of
the variance of the stochastic estimate of the gradients or subgradients (namely, $\alpha_n=0$ in assumption (A2)). One exception is 
\cite{Barty07}, dealing with a stochastic projected subgradient algorithm on a Hilbert space. 
Weak almost sure convergence of the iterates generated by the stochastic forward-backward splitting
algorithm can be derived from the general results in  \cite{ComPes14}, 
without additional monotonicity assumptions on $A$ or $B$.   The assumptions in \cite[Proposition 5.7]{ComPes14}
allows for a nonvanishing stepsize but requires summability of the stochastic errors. 
Even in the minimization case, we are not aware of any paper proving convergence of the stochastic forward-backward algorithm 
with constant step-size without assuming that the variance of the stochastic estimate goes to zero. 

\begin{remark} \label{t:2nviv}
Under the same assumptions as in Theorem \ref{t:2nv}, suppose in addition that $B$ is strictly monotone at $\overline{w}$. 
The assumptions of Theorem \ref{t:2nv}\ref{t:2nviii} are satisfied 
in the following two cases: 
\begin{enumerate}
\item $\HH$ is finite dimensional. Indeed in this case the weak and strong topology coincide, and therefore $B$ is weakly continuous.
\item $B$ is bounded and linear, since in this case $B$ is weakly continuous. For instance, this covers the case of regularized quadratic minimization on Hilbert spaces.
\end{enumerate}
\end{remark}

\subsection{Nonasymptotic bounds}
In this section we focus on convergence in expectation.  
We provide results for the case when either  $A$ or $B$ is strongly monotone.
We derive a nonasymptotic bound for $\E[\|w_n-\overline{w}\|^2]$
similarly to  what has been done for the stochastic gradient algorithm for the
case of minimization of a smooth function in the finite dimensional case 
\cite[Theorem 1]{bach}. 
In the next theorem we will  consider the following assumption.
\begin{assumption}\label{ass:strcon}
Let $\overline{w}$ be a solution of Problem \ref{inc0}.
Furthermore,  suppose that $A$ is $\nu$-strongly monotone and $B$ is $\mu$-strongly monotone at $\overline{w}$, for some 
$(\nu,\mu) \in \left[0,+\infty\right[^2$ such that $\nu+\mu > 0$.  
\end{assumption}

To state the results more concisely, for every $c\in\RR$, we define the  function 
\begin{equation}
\label{eq:bach1}
\varphi_{c}\colon \left]0,\pinf\right[\to \RR\colon
t\mapsto 
\begin{cases}
(t^{c}-1)/c& \text{if $c \not=0$};\\
\log t& \text{if $c =0$}.
\end{cases}
\end{equation}

\begin{theorem}
\label{t:2}
Let $(\underline{\lambda},\overline{\alpha})\in\RPP^2$ and let $(w_n)_{n\in\mathbb{N}^*}$ be the sequence generated 
by Algorithm \ref{a:maininc}. 
Assume that conditions $(A1), (A2)$, $(A3)$, and Assumption \ref{ass:strcon} are  satisfied
and suppose that  $\inf_{n\in\NN^*} \lambda_n\ge\underline{\lambda}$, $\sup_{n\in\mathbb{N}^*}\alpha_n \leq \bar{\alpha}$, and 
that $\gamma_n =  c_1n^{-\theta}$ for some $\theta \in \left]0,1\right]$ and for some  $c_1\in\,]0,+\infty[$. 
Set 
\begin{equation} 
\label{e:ciao}
t = 1-2^{\theta-1}\geq 0, \quad c=\frac{c_1\underline{\lambda}(2\nu+\mu\varepsilon)}{(1+\nu)^2},\quad \text{and}\quad 
\tau=\frac{2\sigma^2c_1^2 (1+\overline{\alpha} \|B\overline{w}\| )}{c^2 }. 
\end{equation}
Let $n_0$ be the smallest integer such that for every integer $n\geq n_0>1$, it holds $\max\{c,c_1\} n^{-\theta}\leq 1.$
Define
$$(\forall n\in\NN^{*})\quad\quad s_n = \E[\|w_{n}-\overline{w}\|^2].$$ 
Then, for every $n\geq 2n_0$, the following hold:
 \begin{enumerate}
 \item \label{t:2i} Suppose that $\theta\in\left]0,1\right[$. Then
\begin{equation}
\label{eq:Est1}
 s_{n+1} \leq \Big(\tau c^2 \varphi_{1-2\theta}(n)
 + s_{n_0}\exp\Big(\dfrac{cn_{0}^{1-\theta}}{1-\theta}\Big) \Big)
\exp\Big(\dfrac{-ct(n+1)^{1-\theta}}{1-\theta} \Big)
+ \dfrac{\tau 2^{\theta}c}{(n-2)^{\theta}}
\end{equation}
\item \label{t:2ii}
Suppose that $\theta=1$. Then
\begin{equation}
\label{eq:Est11} 
s_{n+1}\leq s_{n_0}\Big(\dfrac{n_0}{n+1}\Big)^{c}+ 
\dfrac{\tau c^2}{(n+1)^{c}}\Big(1+ \dfrac{1}{n_0}\Big)^{c}\varphi_{c-1}(n)\,.
\end{equation}
\item
\label{t:2iii} The sequence $(s_n)_{n\in\NN^*}$ satisfies
\begin{equation} 
\label{eq:Est111}
s_n =\begin{cases} O(n^{-\theta}) &\text{if } \theta\in \left]0,1\right[\\ 
O(n^{-c}) &\text{if } \theta=1,\, c<1\\
O(n^{-1}\log n)  &\text{if } \theta=1,\,c=1\\
O(n^{-1})  &\text{if } \theta=1,\,c>1.
 \end{cases}
\end{equation}
\end{enumerate}
\end{theorem}

Theorem \ref{t:2} implies that, even without assuming $(A4)$, in the strongly monotone case 
there is convergence in quadratic mean for every $\theta\in\left]0,1\right]$.
Regardless of the choice of $\theta$, the constants in \eqref{eq:Est1}
and \eqref{eq:Est11} depend on $c$ and thus on the monotonicity constant of $A+B$. 
By \eqref{eq:Est111}, it follows that the best rate is
obtained with  $\theta=1$, for a choice of $c_1$ ensuring $c>1$. 
The resulting rate of convergence for the iterates is $O(1/n)$, which, in the special case of minimization,
 coincides with the one that can be obtained applying recent accelerated variants of proximal gradient methods.

\subsection{Proofs of the Main Results}\label{sec:pf}
We start with a result characterizing the asymptotic behavior of 
stochastic quasi-Fej\'er monotone sequences. 
The following statement is given in \cite[Lemma 2.3]{Barty07} without a proof. A 
version of  Proposition \ref{p:fejer} in the finite dimensional setting can also be found in 
\cite{Ermol68}.  The concept of stochastic Fej\'er sequences has been revisited and extended
in a Hilbert space setting in the recent preprint \cite{ComPes14}.
For the sake of completeness we provide a proof.
\begin{proposition}
\label{p:fejer}
 Let $S$ be a non-empty closed subset of $\HH$, and
let  $(w_n)_{n\in\NN^{*}}$ be stochastic quasi-Fej\'er monotone with respect to $S$.
Then the following hold.
\begin{enumerate}
\item \label{p:fejeri}
Let $w\in S$. Then, there exist  $\zeta_w\in\RR$ and an integrable random vector $\xi_w\in\HH$
such that 
$\E[\|w_n -w\|^2]\to\zeta_w$ and $\|w_n-w\|^2\to\xi_w$ almost surely.
\item \label{p:fejerii}
 $(w_n)_{n\in\NN^{*}}$ is bounded a.s.
\item \label{p:fejeriii}
The set of weak subsequential limits of $(w_n)_{n\in\NN^{*}}$ is non-empty a.s.
\end{enumerate}
\end{proposition}

\begin{proof} 
 It follows from \eqref{e:fejer} that 
\begin{equation}
\label{e:refejer}
(\forall n\in\NN^{*})(\forall w\in S)\quad \E[\|w_{n+1}-w\|^2] 
\leq \E[\|w_n-w\|^2] + \varepsilon_n \,.
\end{equation}

\ref{p:fejeri}: Since the sequence $(\varepsilon_n)_{n\in\NN^*}$ is summable and
$\E[\|w_1-w\|^2]$ is finite, we 
derive from \eqref{e:refejer} that $(\E[\|w_n-w\|^2])_{n\in\NN^*}$ is a real positive quasi-Fej\'er
sequence \cite[Definition 1.1(2)]{Com01}, and therefore it converges to some
$\zeta_{w}\in \RR$ by \cite[Lemma 3.1]{Com01}. Set 
\begin{equation}
\label{e:rn}
(\forall n\in\NN^*)\quad r_n = \|w_n-w\|^2 + \sum_{k=n}^{\infty}\varepsilon_n.
\end{equation}
Then, it follows from \eqref{e:fejer} that
\begin{alignat}{2}
(\forall n\in\NN^*)\quad \E[r_{n+1}|\mathcal{F}_n] &=  \E[\|w_{n+1}-w\|^2  |\mathcal{F}_n] + \sum_{k=n+1}^{\infty}\varepsilon_n\notag\\
&\leq \|w_n-w\|^2 +  \sum_{k=n}^{\infty}\varepsilon_n\notag\\
&= r_n.
\end{alignat}
Therefore $(r_n)_{n\in\NN}$ is a real supermartingale. Since $ \sup_n \E[\min\{r_n, 0\}]=0 < +\infty$
 by \eqref{e:rn}, $r_n$ converges a.s to 
an integrable random variable \cite[Theorem 9.4]{Met82}, that we denote by $\xi_{w}$.

\ref{p:fejerii}\&\ref{p:fejeriii}: Follow directly by \ref{p:fejeri}.\
\end{proof}

For the convenience of the reader, we recall the following well-known property of the 
resolvent of a maximal monotone operator.

\begin{lemma}{\rm\cite[Proposition 23.7]{livre1}}
 \label{l:firm} 
Let $A\colon\HH\to 2^{\HH}$ be maximally monotone. Then, the resolvent of  $A$ is firmly-nonexpansive, i.e., 
\begin{equation}
\label{e:firm2}
 (\forall w\in\HH)(\forall u\in\HH)\quad \|J_A w-J_A u\|^2 
\leq \|u-w\|^2 - \|(w- J_A w)-(u- J_A u)\|^2.
\end{equation}
\end{lemma}
 
We next prove Proposition~\ref{p:1}. 

\begin{proof}[of Proposition~\ref{p:1}]
Let $n\in\NN^*$. Since $\overline{w}$ is a solution of Problem \ref{inc0} we have
\begin{equation}
\overline{w} = 
J_{\gamma_n A}(\overline{w}-\gamma_nB\overline{w})\,.
\end{equation}
It follows from \eqref{e:main1*} and the convexity of $\|\cdot\|^2$ that 
\begin{alignat}{2}
\label{e:est1}
 \|w_{n+1}-\overline{w}\|^2 
&= \|(1-\lambda_n)(w_n-\overline{w}) + \lambda_n(y_n-\overline{w})\|^2\notag\\
&\leq (1-\lambda_n)\|w_n-\overline{w}\|^2
+ \lambda_n \|y_n-\overline{w}\|^2.
\end{alignat}
Since $J_{\gamma_n A}$ is firmly non-expansive by Lemma \ref{l:firm}, setting
\begin{equation}
\label{e:u}
u_n = w_n-y_n -\gamma_n(\SG_n-B\overline{w}).
\end{equation}
we have 
\begin{alignat}{2}
\label{e:est2}
 \|y_n-\overline{w}\|^2
 &\leq \|(w_n-\overline{w}) - \gamma_n (\SG_n - B\overline{w}) \|^2-\|u_n\|^2\notag\\
\notag&= \|w_n-\overline{w}\|^2 
-2\gamma_n\scal{w_n-\overline{w}}{\SG_n - B\overline{w}}\\
&\quad+ \gamma^{2}_n\|\SG_n - B\overline{w} \|^2-\|u_n\|^2,
\end{alignat}
Since $\E[\|\SG_n\|^2]<+\infty$ by assumption, we derive that $\E[\|\SG_n - B\overline{w}\|^2]<+\infty$. 
On the other hand, by induction we get that $\E[\|w_n-\overline{w}\|^2]<+\infty$ and hence $\E[\|w_n-\overline{w}\|]<+\infty$ 
and therefore $\E[\left|\scal{w_n-\overline{w}}{\SG_n - B\overline{w}}\right|] <+\infty$,
so that $\E[\scal{w_n-\overline{w}}{\SG_n - B\overline{w}}|\mathcal{F}_n]$ is well-defined. 
Assumption (A1) yields
\begin{alignat}{2}
\label{e:est3}
\E[\scal{w_n-\overline{w}}{\SG_n - B\overline{w}}] 
&= \E[\E[\scal{w_n-\overline{w}}{\SG_n - B\overline{w}}|\mathcal{F}_n ]\notag\\
&= \E[\scal{w_n-\overline{w}}{\E[ \SG_n - B\overline{w}|\mathcal{F}_n]}]\notag\\
&=\E[\scal{w_n-\overline{w}}{Bw_n - B\overline{w}}].
\end{alignat}
Moreover, using assumption (A2) and the cocoercivity of $B$, we  have 
 \begin{alignat}{2}
\label{e:est4}
 &\E[\|\SG_n - B\overline{w} \|^2]=\notag\\
 &  =\E[\|Bw_n - B\overline{w} \|^2]
+\E[\|\SG_n - Bw_n\|^2]+2\E[\langle Bw_n - B\overline{w}, \SG_n - Bw_n \rangle] \notag\\
&=\E[\|Bw_n - B\overline{w} \|^2] +\E\left[\E[\|\SG_n - Bw_n\|^2|\mathcal{F}_n]\right]+2\E\left[\E[\langle Bw_n - B\overline{w}, \SG_n - Bw_n \rangle|\mathcal{F}_n]\right] \notag\\
&\leq \E[\|Bw_n - B\overline{w} \|^2] +\sigma^2(1+\alpha_n\E[\|Bw_n\|^2])+2\E\left[\langle Bw_n - B\overline{w},\E[\SG_n - Bw_n |\mathcal{F}_n\rangle]\right]\notag\\
&\leq (1+ 2\sigma^2\alpha_n)\E[\|Bw_n - B\overline{w} \|^2]
+ \sigma^2(1 + 2\alpha_n\|B\overline{w}\|^2)\notag\\
&\leq  
\frac{ 1 + 2\sigma^2\alpha_n}{\beta}\E[\scal{w_n-\overline{w}}{Bw_n - B\overline{w}}] 
+ 2\sigma^2(1+ 2\alpha_n\|B\overline{w}\|^2).
 \end{alignat}
Recalling the definition of $\varepsilon$,  from \eqref{e:est1}, \eqref{e:est2}, \eqref{e:est3}, and \eqref{e:est4} we get that 
\begin{alignat}{2}
\label{e:cons}
 \E[\|w_{n+1} -\overline{w} \|^2] 
&\leq (1-\lambda_n) \E[\|w_n-\overline{w}\|^2] + \lambda_n \E[\|y_n-\overline{w}\|^2] \notag\\
&\leq \E[\|w_n-\overline{w}\|^2] 
- \gamma_n\lambda_n\bigg(2-\frac{\gamma_n( 1 + 2\sigma^2\alpha_n)}{\beta}\bigg)\cdot\notag\\
&\quad \cdot\E[\scal{w_n-\overline{w}}{Bw_n-B\overline{w}}] + 2\sigma^2\chi^{2}_n- \lambda_n \E[\|u_n\|^2]\notag\\
&\leq \E[\|w_n-\overline{w}\|^2] 
- \varepsilon \gamma_n\lambda_n \E[\scal{w_n-\overline{w}}{Bw_n-B\overline{w}}] 
 + 2\sigma^2\chi^{2}_n -\lambda_n \E[\|u_n\|^2].
\end{alignat}

\ref{p:1i}:
Since the sequence $(\chi^{2}_n)_{n\in\NN^{*}}$ is summable by assumption $(A4)$, we derive 
from \eqref{e:cons} that 
$(\E[\|w_{n+1}-\overline{w} \|^2])_{n\in\NN^{*}}$ converges to a finite value.

\ref{p:1iii}: It follows from \eqref{e:cons} that 
\begin{equation}
\label{e:ese1}
\sum_{n\in\NN^{*}}
\gamma_n\lambda_n \E[\scal{w_n-\overline{w}}{Bw_n-B\overline{w}}] < +\infty.
\end{equation}
Since $\sum_{n\in\NN^{*}}\lambda_n\gamma_n = +\infty$ by $(A4)$, we obtain,
\begin{equation}
 \varliminf_{n\to\infty}\E[\scal{w_n-\overline{w}}{Bw_n-B\overline{w}}] = 0
 \end{equation}
which implies, by cocoercivity, $\varliminf_{n\to\infty}\E[\|Bw_n-B\overline{w} \|^2] =0$.

Since $B$ is cocoercive, it is Lipschitzian. Therefore, by  \ref{p:1i}, there exists $M\in\RPP$ such that
\begin{alignat}{2}
(\forall n\in\NN^{*})\quad
 \E[\scal{w_n-\overline{w}}{ Bw_n-B\overline{w}}] \leq \beta^{-1} 
\E[\|w_n-\overline{w}\|^2] \leq M .
\end{alignat}
Hence, we derive from (A4) and  \eqref{e:est4}  that
\begin{equation}
\label{e:abc1}
 \sum_{n\in\NN^{*}} \lambda_n\gamma^{2}_n 
\E[\|\SG_n-B\overline{w}\|^2] < +\infty.
\end{equation}

\ref{p:1ii} It follows from \eqref{e:cons} that  $\sum_{n\in\NN^{*}}
\gamma_n\lambda_n \E[\|u_n\|^2] < +\infty$.
Finally, by \eqref{e:abc1} we obtain
\begin{equation}
 \sum_{n\in\NN^{*}}\lambda_n\E[\|w_n-y_n\|^2] \leq 2\sum_{n\in\NN} \lambda_n
\E[\|u_n\|^2]
+ 2\sum_{n\in\NN^{*}} \lambda_n\gamma^{2}_n 
\E[\|\SG_n-B\overline{w}\|^2] < +\infty.
\end{equation}
Therefore, \ref{p:1ii} is proved.
\end{proof}

Next we prove Theorem~\ref{t:2nv}, which is based on Propositions \ref{p:fejer} and \ref{p:1}.

\begin{proof}[Theorem~\ref{t:2nv}]
\ref{t:2nvi-}
Let $n\in\NN^*$. Reasoning as in the proof of Proposition \ref{p:1}, we have 
\begin{alignat}{1}
\label{e:alm1}
\|y_n&-\overline{w}\|^2\\
&\leq \|w_n-\overline{w}\|^2\! -2\gamma_n\scal{w_n-\overline{w}}{\SG_n-B\overline{w}} + \gamma^{2}_n\|\SG_n - B\overline{w} \|^2 -\|u_n\|^2,
\end{alignat}
where $u_n=w_n-y_n -\gamma_n(\SG_n-B\overline{w})$  is defined as in \eqref{e:u}.

We next estimate the conditional expectation with respect to $\mathcal{F}_n$ 
 of each term in the right hand side of \eqref{e:alm1}. 
Since $w_n$ is $\mathcal{F}_n$-measurable, we have 
\begin{equation}
 \E[\|w_n-\overline{w}\|^2| \mathcal{F}_n] = \|w_n-\overline{w}\|^2,
\end{equation}
and using  condition (A1),
\begin{alignat}{2}
\label{e:alm2}
 \E[\scal{w_n-\overline{w}}{\SG_n 
- B\overline{w}}| \mathcal{F}_n] 
&= \scal{w_n-\overline{w}}{\E[ \SG_n - B\overline{w}| \mathcal{F}_n }\notag\\
&=\scal{w_n-\overline{w}}{Bw_n - B\overline{w}}. 
\end{alignat}
Noting that $Bw_n$ is $\mathcal{F}_n$-measurable since 
$w_n$ is $\mathcal{F}_n$-measurable and $B$ is continuous, and using condition (A2), 
we derive
 \begin{alignat}{2}
\label{e:alm3}
 &\E[\|\SG_n - B\overline{w} \|^2| \mathcal{F}_n]= \E[\|\SG_n - Bw_n\|^2 |\mathcal{F}_n  ]+\E[\|Bw_n - B\overline{w} \|^2| \mathcal{F}_n]+ \E[\langle Bw_n - B\overline{w}, \SG_n - Bw_n\rangle |\mathcal{F}_n ]\notag\\
&\leq \sigma^2(1+\alpha_n\|Bw_n\|^2)+\|Bw_n - B\overline{w} \|^2 \notag\\
&\leq  \|Bw_n - B\overline{w} \|^2+
\sigma^2(1+ 2\alpha_n  \|Bw_n - B\overline{w} \|^2+ 2\alpha_n\|B\overline{w}\|^2 )\notag\\
&\leq \frac{(1 + 2\sigma^2\alpha_n )}{\beta} \scal{w_n-\overline{w}}{Bw_n - B\overline{w}} 
+ \sigma^2(1+ 2\alpha_n\|B\overline{w}\|^2),
 \end{alignat}
where the last inequality follows from the cocoercivity
 of $B$.
Now, note that by convexity we have
\begin{alignat}{2}
\label{e:alm4}
\|w_{n+1}-\overline{w}\|^2 
&\leq (1-\lambda_n)\|w_n-\overline{w}\|^2
+ \lambda_n \|y_n-\overline{w}\|^2.
\end{alignat}
Taking the conditional expectation and invoking \eqref{e:alm1}, \eqref{e:alm2}, 
 \eqref{e:alm3}, we obtain, 
\begin{alignat}{2}
\label{e:concl2}
 \E[&\|w_{n+1}-\overline{w} \|^2|\mathcal{F}_n] 
\leq (1-\lambda_n) \|w_n-\overline{w}\|^2 + \lambda_n \E[\|y_n-\overline{w}\|^2|\mathcal{F}_n ] \notag\\
&\leq \|w_n-\overline{w}\|^2
- \gamma_n\lambda_n\bigg(2-\frac{\gamma_n(1+2\sigma^2\alpha_n)}{\beta}\bigg)\scal{Bw_n-B\overline{w}}{ w_n-\overline{w}}
+ 2\sigma^2\chi^{2}_n-\lambda_n\E[ \|u_n\|^2|\mathcal{F}_n]\notag\\
&\leq \|w_n-\overline{w}\|^2
- \varepsilon \gamma_n\lambda_n \scal{Bw_n-B\overline{w}}{ w_n-\overline{w}} + 
2\sigma^2\chi^{2}_n-\lambda_n\E[ \|u_n\|^2|\mathcal{F}_n].
\end{alignat}
Hence $(w_n)_{n\in\NN^{*}}$ is stochastic quasi-Fej\'er monotone with 
respect to the set
$\zer(A+B)$, which is nonempty, closed, and convex. 

\ref{t:2nvi}: It follows from Proposition \ref{p:fejer}\ref{p:fejeri} that
$(\|w_n-\overline{w}\|^2)_{n\in\NN^{*}}$ converges a.s 
to some integrable random variable $\zeta_{\overline{w}}$. 

\ref{t:2nvii} Since $B$ is uniformly monotone at $\overline{w}$,
there exists an increasing function 
$\phi\colon\left[0,+\infty\right[\to \left[0,+\infty\right[$ 
vanishing only at $0$ such that 
\begin{equation}
\label{e:unifm}
 \scal{Bw_n-B\overline{w}}{w_n-\overline{w}} \geq \phi(\|w_n-\overline{w}\|).
\end{equation}
and thus $\overline{w}$ is the unique
solution of Problem \ref{inc0}.
We derive from Proposition \ref{p:1} \ref{p:1iii} and \eqref{e:unifm} that 
\begin{equation}
\sum_{n\in\NN^{*}} \lambda_n\gamma_n\E[\phi(\|w_n-\overline{w}\|)]  < \infty,
\end{equation}
 and hence 
\begin{equation}
 \sum_{n\in\NN^{*}} 
\lambda_n\gamma_n\phi(\|w_n-\overline{w}\|)  < \infty\quad \text{a.s.}
\end{equation}
Since  $(\lambda_n\gamma_n)_{n\in\NN^{*}}$ is  not summable by (A4),
we have $\varliminf\phi(\|w_n-\overline{w}\|) =0$ a.s.
Consequently, taking into account  \ref{t:2nvi}, there exist $\Omega_1\subset \Omega$ and an integrable
random variable $\zeta_{\overline{w}}$ in $\HH$ such that
$P(\Omega_1)=1$, and, for every $\omega\in\Omega_1$,
$\varliminf\phi(\|w_n(\omega)-\overline{w}\|) =0$ and $\|w_n(\omega)-\overline{w}\|^2\to \zeta_{\overline{w}}$. 
Let $\omega\in\Omega_1$.
Then, there exists a subsequence $(k_n)_{n\in\NN^{*}}$ 
such that $ \phi(\|w_{k_n}(\omega)-\overline{w}\|)\to0$,
which implies that $\|w_{k_n}(\omega)-\overline{w}\| \to 0$, and
therefore $w_{n}(\omega)\to \overline{w}$. Since $\omega$ is 
arbitrary in $\Omega_1$, the statement follows.

\ref{t:2nviii}:
By Proposition \ref{p:1}\ref{p:1i}, 
$\varliminf \E[\|Bw_n-B\overline{w}\|^2]= 0$, and hence there exists 
a subsequence $(k_n)_{n\in\NN^{*}}$ such that 
\begin{equation}
\label{limi}
 \lim_{n\to\infty}\E[\|Bw_{k_n}-B\overline{w}\|^2]= 0.
\end{equation}
 Therefore,  there
exists a subsequence $(p_{n})_{n\in\NN^{*}}$ of $(k_n)_{n\in\NN^{*}}$ such that 
\begin{equation}
\label{e:gen1}
 \|Bw_{p_n}-B\overline{w}\|^2 \to 0 
\quad \text{almost surely}. 
\end{equation}
Thus, it follows from \ref{t:2nvi} and Proposition \ref{p:fejer}\ref{p:fejeriii} that there exists $\Omega_1\in\boldsymbol{\mathcal{A}}$ 
such that $\boldsymbol{\mathcal{P}}(\Omega_1)=1$ and, for every $\omega\in\Omega_1$, 
$(w_n(\omega))_{n\in\NN^*}$ has  weak cluster points and $\|Bw_{p_n}(\omega)-B\overline{w}\|^2 \to 0 
$.
Fix $\omega\in\Omega_1$ and let $\overline{z}(\omega)$ be a weak cluster point of $(w_{p_n}(\omega))_{n\in\NN^{*}}$, 
then there exists a subsequence 
$(w_{q_{p_n}}(\omega))_{n\in\NN^{*}}$ 
 such that
$ w_{q_{p_n}}(\omega)\weakly \overline{z}(\omega) $.
Since $B$ is weakly continuous,
 $Bw_{q_{p_n}}(\omega) \weakly B\overline{z}(\omega)$. 
Therefore,  $B\overline{w} = B\overline{z}(\omega)$,
and hence
$\scal{B\overline{z}(\omega)-B\overline{w}}{\overline{z}(\omega)-\overline{w}}  =0$. 
Since $B$ is strictly monotone at $\overline{w}$, we obtain, $\overline{w} = \overline{z}(\omega)$. 
This shows that $w_{q_{p_n}}(\omega) \weakly \overline{w}$. Defining $(t_n)_{n\in\NN^*}$ by 
setting, for every $n\in\NN^*$, $t_n=q_{p_n}$ the statement follows. 
\end{proof}

The following lemma establishes a non asymptotic bound for numerical sequences satisfying a 
given recursive inequality. This is a non asymptotic version of Chung's lemma \cite[Chapter 2, Lemma 5]{Pol87}
(see also \cite{bach}).

\begin{lemma}
\label{l:ocs}
 Let $\alpha$ be in $\left]0,1\right]$, 
and let $c$ and $\tau$ be in $]0,+\infty[$, 
let  $(\eta_n)_{n\in\NN^{*}}$ be the
sequence defined by $(\forall n\in \NN^*)\; \eta_n = c n^{-\alpha}$. 
Let $(s_{n})_{n\in\NN^*}$ be such that 
\begin{equation}
\label{e:iter}
(\forall n\in \NN^*)\quad 0 \leq s_{n+1} \leq (1-\eta_n) s_n + \tau\eta_{n}^2.
\end{equation}
Let $n_0$ be the smallest integer such that $(\forall n\geq n_0>1 )\; \eta_n\leq 1$, 
set $t=  1-2^{\alpha-1} \geq  0$, and define $\varphi_{1-2\alpha}$ and $\varphi_{c-1}$ 
as in \eqref{eq:bach1}.
Then, for every $n\geq 2n_0$,  if $\alpha\in\left]0,1\right[$,
\begin{equation}
 s_{n+1} \leq  \Big(\tau c^2 \varphi_{1-2\alpha}(n)
 + s_{n_0}\exp\Big(\dfrac{cn_{0}^{1-\alpha}}{1-\alpha}\Big) \Big)
\exp\Big(\dfrac{-ct(n+1)^{1-\alpha}}{1-\alpha} \Big)
+ \dfrac{\tau 2^{\alpha}c}{(n-2)^{\alpha}}
\end{equation}
 and if $\alpha=1$
\begin{equation}
s_{n+1}\leq  s_{n_0}\Big(\dfrac{n_0}{n+1}\Big)^{c}+ 
\dfrac{\tau c^2}{(n+1)^{c}}\Big(1+ \dfrac{1}{n_0}\Big)^{c}\varphi_{c-1}(n).
\end{equation}
\end{lemma}
\begin{proof}  
Note that, for every $n\in\NN^*$ and for every integer $m\leq n$:
\begin{equation} \label{eq:varineq}
 \sum_{k=m}^n k^{-\alpha} \geq \varphi_{1-\alpha}(n+1) -\varphi_{1-\alpha} (m),
\end{equation}
where $\varphi_{1-\alpha}$ is defined by \eqref{eq:bach1}.
Since all terms in \eqref{e:iter} are positive for $n\geq n_0$, 
by applying the recursion $n-n_0$ times we have
 \begin{equation}\label{eq:rec}
 s_{n+1} \leq s_{n_0}\prod_{k=n_0}^n(1-\eta_k)  
+\tau \sum_{k=n_0}^n \prod_{i=k+1}^n(1-\eta_i)\eta_{k}^2. 
\end{equation}
Let us  estimate the first term in the right hand side of \eqref{eq:rec}. 
Since  $1-x\leq \exp(-x)$  for every $x\in\mathbb{R}$, from  \eqref{eq:varineq}, we derive
\begin{alignat}{2}
\label{e:conca}
\nonumber s_{n_0}\prod_{k=n_0}^n\left(1-\eta_k\right)  
&= s_{n_0} \prod_{k=n_0}^n\Big(1-\frac{c}{k^{\alpha}}\Big)
\leq s_{n_0} \exp\left(-c\sum_{k=n_0}^n k^{-\alpha} \right)\\ & \leq\begin{cases}
    s_{n_0}\Big(\dfrac{n_0}{n+1}\Big)^{c}& \text{if $\alpha = 1$},\\
    \\[-2ex]
s_{n_0} \exp\Big(\dfrac{c}{1-\alpha}(n_{0}^{1-\alpha} - (n+1)^{1-\alpha}) \Big)
& \text{if $0< \alpha < 1$.}
   \end{cases}
\end{alignat}
To estimate the second term on the right hand side of \eqref{eq:rec}, let us first consider the case $\alpha < 1$,
and let $m\in\mathbb{N}^*$ such that $n_0\leq n/2 \leq  m+1 \leq (n+1)/2$. We have 
\begin{alignat}{2}
 \sum_{k= n_0}^n& \prod_{i=k+1}^n(1-\eta_i)\eta_{k}^2 =\sum_{k= n_0}^m \prod_{i=k+1}^n(1-\eta_i)\eta_{k}^2 +\sum_{k=m+1}^n \prod_{i=k+1}^n(1-\eta_i)\eta_{k}^2 \notag\\
&\leq \exp\big(-\sum_{i=m+1}^n\eta_i\big)\sum_{k=n_0}^m\eta_{k}^{2} + {\eta_m} \sum_{k=m+1}^n \left(\prod_{i=k+1}^n(1-\eta_i)-\prod_{i=k}^n(1-\eta_i)\right)\notag\\ 
&=\exp\big(-\sum_{i=m+1}^n\eta_i\big)\sum_{k=n_0}^m\eta_{k}^{2} + {\eta_m}  \left(1-\prod_{i=m+1}^n(1-\eta_i)\right)\notag\\ 
&\leq \exp\big(-\sum_{i=m+1}^n\eta_i\big)\sum_{k=n_0}^m\eta_{k}^{2} +\eta_m\notag\\
&\leq 
c^2\exp\Big(\frac{c}{1-\alpha}( (m+1)^{1-\alpha} - (n+1)^{1-\alpha}) \Big)
\varphi_{1-2\alpha}(n) + \eta_m\\
\label{eq:rhss}&\leq 
c^2\exp\Big(\frac{-ct(n+1)^{1-\alpha}}{1-\alpha} \Big)
\varphi_{1-2\alpha}(n) + \frac{2^{\alpha} c }{\mu(n-2)^{\alpha}}.
\end{alignat}
Hence, combining \eqref{e:conca} and \eqref{eq:rhss}, for $\alpha\in\left]0,1\right[$ we get
\begin{alignat}{2}
 \quad 
s_{n+1} &\leq\! \Big(\tau c^2 \varphi_{1-2\alpha}(n)
 + s_{n_0}\exp\!\Big(\frac{cn_{0}^{1-\alpha}}{1-\alpha}\Big)\! \Big)
\exp\!\Big(\frac{-ct(n+1)^{1-\alpha}}{1-\alpha} \Big)
+ \frac{\tau 2^{\alpha}c}{(n-2)^{\alpha}}
\end{alignat}
We next estimate the second term  in the right hand side of \eqref{eq:rec} in the case $\alpha =1$. We have
\begin{alignat}{2}
 \notag \sum_{k= n_0}^n \prod_{i=k+1}^n(1-\eta_i)\eta_{k}^2 
&=  \frac{c^2}{(n+1)^{c}}\Big(1+ \frac{1}{n_0}\Big)^{c} 
\sum_{k= n_0}^{n}\frac{1}{k^{2-c}} \\
&\leq   \frac{ c^2}{(n+1)^{c}}\Big(1+ \frac{1}{n_0}\Big)^{c}\varphi_{c-1}(n).
\end{alignat}
Therefore,  for $\alpha =1$, we obtain,
\begin{equation}
 s_{n+1} \leq s_{n_0}\Big(\frac{n_0}{n+1}\Big)^{c}+
 \frac{\tau c^2}{(n+1)^{c}}\Big(1+ \frac{1}{n_0}\Big)^{c} \varphi_{c-1}(n),
\end{equation}
which completes the proof.
\end{proof}

We are now ready to prove Theorem \ref{t:2}.
\begin{proof}(Theorem \ref{t:2}) 

Since  $\mu+\nu>0$, then $A+B$ is strongly monotone at $\overline{w}$. 
Hence, problem \eqref{e:inc0} has a unique solution, 
i.e, $\zer(A+B)  =\{\overline{w}\}$. Let $n\in\NN^*$.
Since $\gamma_n A$ is $\gamma_n\nu$-strongly monotone, 
 by \cite[Proposition 23.11]{livre1} $J_{\gamma_n A}$ is $(1+\gamma_n\nu)$-cocoercive, 
 and then
\begin{align}
\|y_n - \overline{w}\|^2
=\|J_{\gamma_nA}(w_n-\gamma_n \SG_n)-J_{\gamma_nA}(\overline{w}-\gamma_n B\overline{w})\|^2 \\
  \leq \frac{1}{(1+\gamma_n\nu)^2}
\|(w_n-\overline{w}) - \gamma_n (\SG_n - B\overline{w}) \|^2.
\end{align}
Next, proceeding as in the proof of Proposition \ref{p:1} and recalling
\eqref{e:est3}-\eqref{e:est4}, we obtain
\begin{align}
 \nonumber\E[\|y_n -\overline{w}\|^2] \leq &\frac{1}{(1+\gamma_n\nu)^2}
\left(\E[\|w_n-\overline{w}\|^2]-\gamma_n\left(2-\gamma_n\frac{1+2\sigma^2\alpha_n}{\beta}\right)\cdot \right.\\
 \label{eq:ymenw}& \cdot\E[\scal{w_n-\overline{w}}{Bw_n-B\overline{w}}]
+2\gamma_n^2\sigma^2 (1+\alpha_n\|B\overline{w}\|^2)\bigg).
\end{align}
Since $B$ is strongly monotone of parameter $\mu$ at $\overline{w}$, 
\begin{equation}\label{eq:strmono}
 \scal{Bw_n-B\overline{w}}{w_n-\overline{w}}\geq \mu\|w_n-\overline{w}\|^2\,.
\end{equation} 
Therefore, from \eqref{eq:ymenw},
recalling the definition of $\varepsilon$ in $(A3)$, we get
 \begin{align}
\label{eq:strmon}
\lambda_n\E[\|y_n -\overline{w}\|^2] & \leq 
 \frac{\lambda_n}{(1+\gamma_n\nu)^2}\bigg((1-\gamma_n\mu\epsilon)\E[\|w_n-\overline{w}\|^2]
+ 2\sigma^2\chi_n^2\bigg).
\end{align}
Hence, by definition of $w_{n+1}$,
\begin{align}
\label{e:consq}
 \E[\|w_{n+1}-\overline{w} \|^2]  
&\leq \bigg(1-\frac{\lambda_n\gamma_n(2\nu + 
\gamma_{n}\nu^2+\mu\epsilon)}{(1+\gamma_n\nu)^2}\bigg)\E[\|w_n-\overline{w}\|^2]+  
\frac{ 2\sigma^2\chi^{2}_n}{(1+\gamma_n\nu)^2}.
\end{align}
Now, suppose that $n\geq n_0$. 
Since $\gamma_n\leq \gamma_{n_0}=c_1 n_0^{-\theta}\leq 1$, we have
\begin{equation}\label{eq:l1}
\frac{\lambda_n\gamma_n(2\nu + \gamma_{n}\nu^2+2\mu\epsilon)}{(1+\gamma_n\nu)^2} \geq 
\frac{\underline{\lambda}(2\nu+\mu\varepsilon)}{(1+\nu)^2} \gamma_n=c n^{-\theta}.
\end{equation}
On the other hand, 
\begin{equation}\label{eq:l2}
\frac{ 2\sigma^2\chi^{2}_n}{(1+\gamma_n\nu)^2}
\leq {2\sigma^2 (1+\overline{\alpha} \|B\overline{w}\|^2)} c_1^2n^{-2\theta}\,. 
\end{equation}
Then, putting together \eqref{e:consq}, \eqref{eq:l1}, and \eqref{eq:l2}, we get
\begin{align}
\label{e:consq2}
 \E[\|w_{n+1}-\overline{w} \|^2]  
&\leq (1-\eta_n)\E[\|w_n-\overline{w}\|^2]+ \tau\eta_n^2,
\end{align}
with $\tau=2\sigma^2c_1^2 (1+\overline{\alpha} \|B\overline{w}\|^2 )/c^2$ and $\eta_n=c n^{-\theta}$.

\ref{t:2i}\&\ref{t:2ii}: Inequalities \eqref{eq:Est1} and \eqref{eq:Est11}  follow from \eqref{e:consq2} by
applying Lemma \ref{l:ocs}.

\ref{t:2iii} \eqref{eq:Est111}
Let $\theta\in\left]0,1\right[$. Then \eqref{eq:Est1} yields $\E[\|w_{n+1}-\overline{w}\|^2]=O(n^{-\theta})$.
Let $\theta=1$. Then \eqref{eq:Est11} implies $\E[\|w_{n+1}-\overline{w}\|^2]=O(n^{-c})+O(n^{-\theta}\varphi_{c-1}(n))$.
If $c\neq 1$, it follows from \eqref{eq:bach1} that $\varphi_{c-1}(n)=O(n^{c-1})$, and 
in this case $\E[\|w_{n+1}-\overline{w}\|^2]=O(n^{-c})+O(n^{-1})$. 
If $c=1$, then it follows again from \eqref{eq:bach1} that  $\E[\|w_{n+1}-\overline{w}\|^2]=O(n^{-1})+O(n^{-1}\log n)$.
\end{proof}

\section{Special cases}\label{sec:sc}

In this section, we study two special instances of Problem \ref{inc0}, namely
 variational inequalities and minimization problems. Moreover, for variational inequalities, we
prove an additional result, showing that a suitably defined merit function \cite{Aus76} goes to zero when evaluated 
on the iterates of the stochastic forward-backward algorithm. This merit function has been used to 
quantify the inaccuracy of an approximation of the solution in \cite{Jud11}.

\subsection{Variational Inequalities}
In this section we focus on a special case of Problem \ref{inc0}, assuming that $A$ is the 
subdifferential of $G\in\Gamma_0(\HH)$.   

\begin{problem}
\label{varine}
 Let $B\colon\HH\to\HH$ be a $\beta$-cocoercive operator, for some 
$\beta \in \left]0,+\infty\right[$,
let $G$ be a function in $\Gamma_0(\HH)$. 
The problem is to solve the following  variational inequality \cite{Lion67,Zei90,livre1}
\begin{equation}
\label{e:varine}
 \text{find $\overline{w}\in\HH $ such that }
\quad (\forall w\in\HH)\quad \scal{\overline{w} -w}{B\overline{w}} + G(\overline{w}) \leq G(w),
\end{equation}
under the assumption that \eqref{e:varine} has at least one solution.
\end{problem}
There is a line of research studying stochastic algorithms for variational 
inequalities on finite dimensional spaces. The sample average approximation  has been studied e.g. 
in \cite{Sha03,CheWetZha12} (see also references therein), and a mirror proximal stochastic approximation algorithm to solve variational 
inequalities corresponding to a maximal monotone
operator has been proposed in \cite{Jud11}. A stochastic iterative proximal method has been proposed in \cite{KosNedSha13}, and almost 
sure convergence properties of a stochastic forward-backward  splitting algorithm for solving strongly monotone variational inequalities 
has been studied in \cite{JiaXu08}.

In the setting of Problem \ref{varine},  let $(\SG_n)_{n\in\NN^*}$ be a random process taking 
values in $\HH$ such that 
$(\forall n\in\NN^*)\; \E[\| \SG_n\|^2]  < +\infty$, and 
$(\forall n\in\mathbb{N}^*)$ $\gamma_n=c_1n^{-\theta}$ for some  $\theta \in \left]0,1\right]$ and for some  $c_1\in\,]0,+\infty[$. 
Let $(\lambda_n)_{n\in\NN^*}$ be a sequence in $\left]0,1\right]$, and let $w_1\colon\Omega \to \HH$ be a random variable such that $\E[\|w_1\|^2]<+\infty$. Set 
\begin{equation}
\label{e:ex11}
(\forall n\in\NN^*)\quad
\begin{array}{l}
\left\lfloor
\begin{array}{l}
z_n = w_n- \gamma_n \SG_n\\
y_{n} = \prox_{\gamma_n G}z_n\\
w_{n+1} = (1-\lambda_n) w_{n} + \lambda_ny_{n}.\\
\end{array}
\right.\\[2mm]
\end{array}
\end{equation}

\begin{corollary}\label{c:1mezzo} Suppose that conditions $(A1),(A2),(A3)$, and $(A4)$ are satisfied. 
Let $(w_n)_{n\in\NN^*}$ be the sequence defined by \eqref{e:ex11}.
Then  
\begin{enumerate}
\item\label{c:2nvii} 
If $B$ is uniformly monotone at $\overline{w}$, 
then $w_n \to \overline{w}$ a.s.
\item \label{c:2nviii}
If $ B$  is strictly monotone at $\overline{w}$ and  weakly continuous, 
then there exists a subsequence $(w_{t_n})_{n\in\NN^{*}}$ such that
 $w_{t_n}\weakly\overline{w}$ a.s.
\item \label{c:2nviv}
 If $B$ is strictly monotone, 
and either $\HH$ is finite dimensional space or $B$ is a bounded and linear, 
 then there exists $\Omega_1\in\boldsymbol{\mathcal{A}}$ such that $\boldsymbol{\PP}(\Omega_1)=1$ such that, for almost every $\omega\in\Omega_1$, there exists a subsequence $(w_{t_n})_{n\in\NN^{*}}$ such that
 $w_{t_n}(\omega)\weakly\overline{w}$.
\end{enumerate}               
\end{corollary}
\begin{proof} The results follow from Theorem \ref{t:2nv}. 
\end{proof}

The following assumption will be used in the next corollary. 
\begin{assumption}\label{ass:strcon2}  Let $\overline{w}$ be a solution of \eqref{e:varine}. Suppose that $G$ is $\nu$-strongly convex and $B$ is $\mu$-strongly monotone at $\overline{w}$ for some  $(\nu,\mu) \in \left[0,+\infty\right[^2$ such that $\nu+\mu > 0$.
\end{assumption}
Note that, while on $B$ we can assume a local strong monotonicity property, on the function $G$ we need a global assumption.

\begin{corollary} 
\label{c:1} 
Let $\underline{\lambda}\in\left]0,1\right]$ and let $\overline{\alpha}\in\left[0,+\infty\right[$. Assume that conditions  (A1), (A2), (A3) and  Assumption \ref{ass:strcon2} are satisfied, with $\sup_{n\in\mathbb{N}^*}\alpha_n \leq \bar{\alpha}$
and $\inf_{n\in\NN^*} \lambda_n\geq \underline{\lambda}$.   Let $(w_n)_{n\in\NN^*}$ be the sequence defined by \eqref{e:ex11} with $\gamma_n=c_1 n^{-\theta}$ for some $c_1\in\left]0,+\infty\right[$ and $\theta\in \left]0,1\right]$. 
Set $t= 1-2^{\theta-1}\geq 0$, 
$c={c_1\underline{\lambda}(2\nu+\mu\varepsilon)}/{(1+\nu)^2} $, 
$\tau={2\sigma^2c_1^2 (1+\overline{\alpha} \|B\overline{w}\| )/c^2 }$ 
and let $n_0$ be the smallest integer such that $(\forall n\geq n_0>1)\; \max\{c,c_1\} n^{-\theta}\leq 1.$ 
Then, by setting $s_n=\E[\|w_n-\overline{w}\|^2]$, the following holds
\begin{enumerate}
\item If $\theta\in\left]0,1\right[$, then, for every $n\in\NN^*$,  $s_{n+1}$ fulfills \eqref{eq:Est1}.
\item If $\theta=1$, then, for every $n\in\NN^*$, $s_{n+1}$ fulfills \eqref{eq:Est11}. 
\item  $(s_n)_{n\in\NN^*}$ satisfies \eqref{eq:Est111}.
\end{enumerate}
\end{corollary}

\begin{proof} Set $A = \partial\nsm$. Then Problem \ref{varine} reduces to a particular case 
of Problem \ref{inc0}. Hence, the result follows from Theorem \ref{t:2}.
\end{proof}

When $G$ is the indicator function of a non-empty, closed, convex subset $C$ of $\HH$, Problem \ref{varine}
reduces to the problem of solving a classic variational inequality \cite{Lion67,Martinet70}, namely to find 
$\overline{w}$ such that 
\begin{equation}\label{eq:varin}
 (\forall w\in C)\quad \scal{B \overline{w}}{\overline{w}-w} \leq 0.
\end{equation}
Proximal algorithms  are often used to solve this problem, see \cite[Chapter 25]{livre1} and references therein. When $B$ is accessible only through a stochastic
oracle, the available methods are limited. In \cite{CheWetZha12} a smoothing sample average approximation method
is analyzed when $B$ can be written as an expectation. An iterative Tikhonov regularization method, based on 
an iterative projected scheme is studied in \cite{KosNedSha13}. Recently, a stochastic mirror-prox algorithm has been proposed 
in \cite{Jud11} for the case when $C$ is a nonempty, compact, convex subset of $\RR^d$, and $B$ is only Lipschitz continuous. 
We also remark that in  \cite{JiaXu08} almost sure convergence of a forward-backward splitting algorithm with respect 
to a noneuclidean metric is studied in  a finite dimensional space, when $B$ is given as an expectation.  

Note that, by \cite[Lemma 1]{Bro65}, since cocoercivity of $B$ implies Lipschitz continuity,  $\overline{w}$ is a solution of \eqref{eq:varin} if and only if 
\begin{equation}
\label{eq:wsol}
(\forall w\in C)\quad \scal{Bw}{\overline{w}-w}\leq 0\,.
\end{equation}
As it has been done in \cite{Jud11}, it is therefore natural to quantify the inaccuracy of a candidate solution $u\in\HH$ by 
the merit function
\begin{equation}
V(u)=\sup_{w\in C} \scal{Bw}{u-w}.
\end{equation}
In particular, note that $(\forall u\in\HH)$ $V(u)\geq 0$ and $V(u)=0$ if and only $u$ is a solution of \eqref{eq:wsol}. We will consider  convergence properties of the following iteration, which differs from the one in Algorithm \ref{a:maininc} only by the averaging step. 

\begin{algorithm} 
\label{a:varine}
Let $C$ be a nonempty bounded closed convex subset of $\HH$.
Let $(\gamma_t)_{t\in\NN^*}$ be a  sequence in $]0,+\infty[$.
Let  $(\lambda_t)_{t\in\NN^*}$ be a sequence in $\left[0,1\right]$,
and let $(\SG_t)_{t\in\NN^{*}}$ be a $\HH$-valued random process
such that $(\forall n\in\NN^*)\; \E[\| \SG_n\|^2]  < +\infty$.
Let $w_1\colon\Omega\to\HH$ be a random variable such that $\E[\|w_1\|^2]<+\infty$ and set
\begin{equation}
\label{e:ex2}
(\forall n\in\NN^{*})\quad
\begin{array}{l}
\left\lfloor
\begin{array}{l}
\operatorname{For} \; t =1,\ldots,n\\
\begin{array}{l}
\left\lfloor
\begin{array}{l}
z_t = w_t- \gamma_t\SG_t  \\
y_t=P_{C}z_t\\
w_{t+1} = (1-\lambda_t)w_t+\lambda_t y_t \\
\end{array}
\right.\\[2mm]
\overline{w}_n = \big(\sum_{t=1}^{n}\gamma_t\lambda_t
w_t\big)/\sum_{t=1}^{n}(\gamma_t\lambda_t). 
\end{array}
\end{array}
\right.\\[2mm]
\end{array}\end{equation}
\end{algorithm}

The next theorem gives an estimate of the function $V$ when evaluated on the expectation of $\overline{w}_n$. Note 
that we do not impose any additional monotonicity property on $B$. 

\begin{theorem}{\rm (Ergodic convergence)} 
\label{t:7*}
In the setting of  problem \eqref{e:varine}, assume that  $\nsm = \iota_C$ 
for some nonempty bounded closed convex set $C$ in $\HH$.
Let $(\overline{w}_n)_{n\in\NN^{*}}$ 
be the sequence generated by Algorithm \ref{a:varine}
and suppose that  conditions $(A1)$, $(A2)$, and $(A3)$ hold.
Set
\begin{equation}
\theta_0 = \sup_{u\in C}\frac 12 \E[\|w_1-u\|^2]
\  \text{and}\  
 \theta_{1,n} =\frac 1 2   \sum_{t=1}^n\big(\lambda_t\gamma_{t}^2(1+ \sigma^2\alpha_t) \E[\|Bw_t \|^2] 
+ \sigma^2 \lambda_t\gamma_{t}^2\big),
\end{equation}
then 
\begin{equation}
\label{e:ergodic}
V( \E[\overline{w}_n]) \leq (\theta_0+\theta_{1,n}) \bigg(\sum_{t=1}^n \lambda_t\gamma_t\bigg)^{-1}. 
\end{equation}
Moreover,  suppose that the condition $(A4)$ is also satisfied. Then,  
\begin{equation}
\label{eq:mertoo}
\lim_{n\to +\infty}V(\E[\overline{w}_n])=0.
\end{equation}
In particular, if $(\forall t\in\NN^*)$ $\lambda_t=1$ and $\gamma_t=t^{-\theta}$ for some $\theta\in\,]1/2,1[$, we get
\begin{equation}
V(\E[\overline{w}_n]) =O(n^{\theta-1}).
\end{equation}
\end{theorem}
\begin{proof} 
Since $C$ is a non-empty closed convex set, $P_C$ is non-expansive,  
and  for every $u\in C$, $u = P_Cu$. Hence, from the convexity of $\|\cdot\|^2$
\begin{alignat}{2}
\label{e:dissc1}
(\forall t\in\NN^*)(\forall u\in C)\quad 
\| w_{t+1} -u \|^{2} &\leq (1-\lambda_t) \|w_t-u\|^2+\lambda_t\|y_t-u\|^2\\
&= (1-\lambda_t) \|w_t-u\|^2+\lambda_t \|P_C(w_t-\gamma_t\SG_t) - P_Cu \|^{2}\notag\\
&\leq (1-\lambda_t) \|w_t-u\|^2+\lambda_t\| w_t  - u-\gamma_t\SG_t \|^{2}\notag\\
&\leq \| w_t-u\|^{2} -2\lambda_t\gamma_t \scal{w_t  - u }{\SG_t }
 + \lambda_t\gamma^{2}_t \|\SG_t \|^{2}.\notag
\end{alignat}
We derive from  conditions $(A1)$ that
\begin{equation}
\E[\scal{w_t-u }{\SG_t }| \mathcal{F}_t] =   
\scal{w_t-u }{ Bw_t}
\end{equation}
and from (A3) that
 \begin{alignat}{2}
 \E[\|\SG_t \|^2| \mathcal{F}_t] 
&\leq \E[\|\SG_t - Bw_t \|^2| \mathcal{F}_t]
+  \E[\|Bw_t\|^2 |\mathcal{F}_t  ]+ 2\E[\scal{\SG_t-Bw_t}{Bw_t}|\mathcal{F}_t] \notag\\
&\leq \|Bw_t \|^2 + \sigma^2(1+\alpha_t\|Bw_t\|^2).
 \end{alignat}
Therefore, \eqref{e:dissc1} and the monotonicity of $B$ yield
\begin{alignat}{1}
2\lambda_t\gamma_t \scal{w_t-u }{Bu} &\leq \|w_{t} -u \|^{2}-
 \E[\| w_{t+1} -u \|^{2} | \mathcal{F}_t] 
\hfill + \lambda_t\gamma_{t}^2(1+ \sigma^2\alpha_t)\|Bw_t\|^2) +\sigma^2\lambda_t\gamma_{t}^2,
\end{alignat}
which implies that 
\begin{alignat}{2}
 2\E[\scal{\overline{w}_n-u}{Bu}]& \leq \bigg(\sum_{t=1}^n \lambda_t\gamma_t\bigg)^{-1} 
\sum_{t=1}^n \bigg(\E[\|w_{t} -u \|^{2}]
 - \E[\| w_{t+1} -u \|^{2} ]  \notag\\
&\quad+\hfil \lambda_t\gamma_{t}^2(1+ \sigma^2\alpha_t) \E[\|Bw_t \|^2] 
+ \sigma^2\lambda_t\gamma_{t}^2)  \bigg)\notag\\
&\leq 
\bigg(\sum_{t=1}^n\lambda_t\gamma_t\bigg)^{-1}\bigg( \E[\|w_{1} -u \|^{2}]+   \sum_{t=1}^n\bigg( \lambda_t\gamma_{t}^2(1+ \sigma^2\lambda_t\alpha_t) \E[\|Bw_t \|^2] 
+ \sigma^2\lambda_t\gamma_{t}^2) \bigg) \bigg).
\end{alignat}
Therefore,
\begin{equation} 
\sup_{u\in C} \E[\scal{\overline{w}_n-u}{Bu}] \leq (\theta_0+\theta_{1,n}) \bigg(\sum_{t=1}^n \lambda_t\gamma_t\bigg)^{-1}, 
\end{equation}
which proves \eqref{e:ergodic}.
Finally, since $C$ is bounded, $\theta_0<+\infty$. 
Now, additionally assume that $(A4)$  is satisfied. 
Then $\sum_{t=1}^{+\infty} \lambda_t\gamma_t=+\infty$, therefore, 
to get \eqref{eq:mertoo}, it is enough to prove that 
$(\theta_0+\theta_{1,n})_{n\in\NN^*}$ is bounded. 
Since we derive from $(A4)$ that 
$\sum_{t=1}^{+\infty} \lambda_t\gamma_t^2<+\infty$ and 
$\sum_{t=1}^{+\infty} \lambda_t\gamma_t^2\alpha_t <+\infty$, 
we are left to prove that $(\E[\|Bw_t\|^2])_{t\in\NN^*}$ is bounded. 
This directly follows from Proposition \ref{p:1}\ref{p:1i}. 
The last assertion of the statement follows from \eqref{e:ergodic} when evaluated for 
$(\forall t\in\NN^*)\;\gamma_t=t^{-\theta}$ and $\lambda_t=1$.
\end{proof}

\begin{remark}
As we mentioned before, when $\dim \HH$ is finite,  $C$ is  a non-empty convex, compact 
subset of $\HH$, and $B$ is bounded, an alternative method to solve Problem \ref{varine} can be found in \cite{Jud11}, 
where $\alpha_n=0$ and the assumption of cocoercivity of $B$ is replaced by the weaker Lipschitz continuity assumption. 
Note that, with respect to forward-backward, the mirror-prox algorithm proposed in \cite{Jud11}, 
requires two projections per iteration, rather than one. With such procedure, in \cite{Jud11}  it is proved that
$\E[V(\overline{w}_n)]$ goes to zero. Note that in general $V(\E[\overline{w}_n])\leq \E[V(\overline{w}_n)]$. 
\end{remark}

\subsection{Minimization problems}
In this section, we specialize  the results in Section \ref{sec:main} to  minimization problems. 
In the special case of composite minimization, stochastic implementations of forward-backward 
splitting algorithms, and more generally of first order methods, received much attention and have been recently 
studied in several papers \cite{Duchi09,AtcForMou14,De11,Lan09,LinChePen14,DucAgaJoh12} for the ease of implement and 
 the low  memory requirement of each iteration.
In particular, \cite{Lan09} proposes an accelerated method and derives  a rate of convergence for  the objective function 
values which is optimal both with respect to the smooth component and  the non-smooth term. Similar accelerated
proximal gradient algorithms have been also studied in the machine learning community, see  \cite{Hu,Xiao},
and  also \cite{bottou2005line,Salev07,Shalev08,zhang2008multi}. 

\begin{problem}
\label{proA2}
Let $\beta\in\,]0,+\infty[$, 
let $\nsm\in\Gamma_0(\HH)$,
and let $\smo\colon\HH\to\RR $ be a convex differentiable function,  
with a $\beta^{-1}$-Lipschitz continuous gradient.
The problem is to
\begin{equation}
\label{e:proA2}
  \underset{w\in\HH}{\text{minimize}}\; \obj(w)=\smo(w)+\nsm(w),
\end{equation}
under the assumption that the set of solution to \eqref{e:proA2} is non-empty.
\end{problem} 

In the setting of Problem \ref{proA2}, let $(\forall n\in\mathbb{N}^*)$ $\gamma_n=c_1n^{-\theta}$ for some  $\theta \in \left]0,1\right]$ and for some  $c_1\in\,]0,+\infty[$. 
Let $(\lambda_n)_{n\in\NN^*}$ be a sequence in $\left]0,1\right]$. 
Let  $(\SG_n)_{n\in\NN^*}$ a random process taking values in $\HH$ such that 
$(\forall n\in\NN^*)\; \E[\| \SG_n\|^2]  < +\infty$,
and let $w_1$ a random variable in $\HH$ such that $\E[\|w_1\|^2]<+\infty$. Define
\begin{equation}
\label{e:ex1}
(\forall n\in\NN^*)\quad
\begin{array}{l}
\left\lfloor
\begin{array}{l}
z_n = w_n- \gamma_n \SG_n\\
y_{n} = \prox_{\gamma_n G}z_n\\
w_{n+1} = (1-\lambda_n) w_{n} + \lambda_ny_{n}.\\
\end{array}
\right.\\[2mm]
\end{array}
\end{equation}

The following results are instances of Corollaries \ref{c:1mezzo} and \ref{c:1},
corresponding to the case  $B=\nabla F$.

\begin{corollary} \label{cor:1}
Let $\underline{\lambda}\in\left]0,1\right]$ and $\overline{\alpha}\in\left]0,+\infty\right[$.
Suppose that  conditions  (A1), (A2) and (A3) are satisfied with $B = \nabla\smo$
and $\sup_{n\in\mathbb{N}^*}\alpha_n \leq \bar{\alpha}$. 
Let $\overline{w}$ be a solution of Problem \ref{proA2}. Suppose that $G$ is $\nu$-strongly 
convex and $\smo$ is $\mu$-strongly convex at $\overline{w}$ for some 
$(\nu,\mu) \in \left[0,+\infty\right[^2$ such that $\nu+\mu > 0$.  
Set $t = 1-2^{\theta-1}\geq 0$, 
$c={c_1\underline{\lambda}(\nu+\mu\varepsilon)}/{(1+\nu)^2} $, 
$\tau={2\sigma^2c_1^2 (1+\overline{\alpha} \|\nabla \smo(\overline{w})\| )/c^2 }$ 
and let $n_0$ be the smallest integer such that $(\forall n\geq n_0>1)\; \max\{c,c_1\} n^{-\theta}\leq 1.$
Let $(w_n)_{n\in\NN^*}$ be the sequence generated by Algorithm \ref{e:ex1} and define, for very $n\in\NN^*$, 
$s_n=\E[\|w_n-\overline{w}\|^2]$. 
Then,
\begin{enumerate}
\item
If $\theta\in]0,1[$, then \eqref{eq:Est1} holds.
\item  If $\theta =1$, then \eqref{eq:Est11} holds.
\item  $(s_n)_{n\in\NN^*}$ satisfies \eqref{eq:Est111}.
\end{enumerate}
\end{corollary}
\begin{proof}
Set $B = \nabla \smo$ and $A = \partial\nsm$. The results follow from Corollary \ref{c:1},
the Baillon-Haddad Theorem \cite[Corollary 18.16]{livre1} and the fact that 
for every $n\in\NN^*$, $J_{\gamma_n A} = \prox_{\gamma_n G}$. 
\end{proof}

In the case when $\nsm $ is the indicator function of a non empty closed convex set 
and $(\forall n\in\NN^*)\; \lambda_n =1$, a similar result on the rate of convergence of
$\E[\|w_n-\overline{w}\|^2]$ has been obtained in \cite{yousefian2012stochastic}, under 
similar assumptions to $(A1),\dots,(A4)$
for the case where $\smo$ is strongly 
convex, under the additional assumption of 
boundedness of the conditional expectations of $(\|\SG_n-\nabla\smo(w_n)\|^2)_{n\in\NN^*}$. 

Corollary \ref{cor:1} is the extension to the nonsmooth case of \cite[Theorem 1]{bach}, in particular,  when $\nsm=0$,
we obtain the same convergence rate. Note however that the assumptions
on the stochastic approximations of the gradient of the smooth part are slightly different.
Algorithm~\ref{e:ex1} is closely related to the FOBOS algorithm studied in \cite{Duchi09}
and the stochastic proximal gradient algorithm in \cite{AtcForMou14}.
The main difference is that these papers consider convergence of the average of the iterates. 
Also  uniform boundedness of the iterations and the subdifferentials are required. 
Our convergence results consider convergence of the iterates with no averaging,
without boundedness assumptions. This is relevant for sparsity based regularization,
where averaging can have a detrimental effect. The asymptotic rate $O(n^{-1})$ 
which we obtain  for the iterates improves  the $O((\log n)/n)$ rate derived from \cite[Corollary 10]{Duchi09} 
for  the average of the iterates and it coincides with the one that can be derived by applying optimal
methods \cite{Lan09}.

\begin{corollary}\label{cor:1mezzo}
Suppose that  conditions  (A1), (A2), (A3) and (A4) are satisfied with $B = \nabla\smo$ and
let $(w_n)_{n\in\NN^*}$ be the sequence generated by Algorithm \ref{e:ex1}.  Then  
\begin{enumerate}
\item\label{c2:2nvii} 
If $\smo$ is uniformly convex at $\overline{w}$, 
then $w_n \to \overline{w}$ a.s.
\item \label{c2:2nviii}
If $ \smo$  is strictly convex at $\overline{w}$ and $\nabla \smo$ is  weakly continuous, 
then there exists a subsequence $(w_{t_n})_{n\in\NN^{*}}$ such that
 $w_{t_n}\weakly\overline{w}$ a.s.
\item \label{c2:2nviv}
 If $\smo$ is strictly convex at $\overline{w}$, 
and either $\HH$ is finite dimensional space or $\smo$ is a bounded and linear, 
 then there exists $\Omega_1\in\boldsymbol{\mathcal{A}}$ with $\boldsymbol{\mathcal{P}}(\Omega_1)=1$
 such that, for every $\omega\in\Omega_1$, there exists a subsequence $(w_{t_n})_{n\in\NN^{*}}$ such that
 $w_{t_n}(\omega)\weakly\overline{w}$.
\end{enumerate}               
\end{corollary}
\begin{proof}
 Set $B = \nabla \smo$ and $A = \partial\nsm$. The results follow from Corollary \ref{c:1mezzo},
the Baillon-Haddad Theorem \cite[Corollary 18.16]{livre1}, and the fact that for every $n\in\NN^*$,
$J_{\gamma_n A} = \prox_{\gamma_n G}$. 
\end{proof}

In the optimization setting, the study of  almost sure convergence has a long history, see e.g. 
 \cite{Rob71,Clark78,Ben90,Chen02}  and  references therein. Recent results on 
almost sure convergence of projected stochastic gradient algorithm can be found in 
\cite{Bennar07,Monnez06}, under rather technical assumptions. Our results are  generalizations   
 of the analysis of the stochastic projected subgradient algorithm in \cite{Barty07}. 

\subsection{Minimization over orthonormal bases} \label{a:2}

We next  describe how to apply Algorithm \ref{e:ex1} to minimization over orthonormal bases. 
This  problem often arises in sparse signal recovery as well as 
learning theory (see e.g. \cite{siam07,Tib96}).
 \begin{problem} 
\label{ex:1} 
Let $\beta$ be in $\left]0,+\infty\right[$, let $\nu$ be in $\left[0,+\infty\right]$,   
assume that $\HH$ is separable, and let $(e_k)_{k\in\NN}$ be an orthonormal base of $\HH$.
Let $(\phi_k)_{k\in\NN}$ be a sequence of functions in $\Gamma_0(\RR_{+})$ such that 
$(\forall k\in\NN)\; \phi_k \geq \phi_k(0)=0$ and set 
\begin{equation}
\label{e:G}
G\colon\HH\to\left[0,+\infty\right]\colon w\mapsto \sum_{k\in\NN}\big( \phi_k(\scal{w}{e_k}) + \frac{\nu}{2} |\scal{w}{e_k}|^2 \big).
\end{equation}
Let $\smo$ be in $\Gamma_0(\HH)$
such that $\smo$ is differentiable, $\mu$ strongly convex for some $\mu\in\left[0,+\infty\right[$,
with a $\beta^{-1}$-Lipschitz continuous gradient.  
The problem is to 
\begin{equation}
\label{e:prob3}
 \underset{w\in\HH}{\text{minimize}} \;
\smo(w) + G(w)
\end{equation}
In the case when $\nu = 0$ and $\smo$ is non-strongly convex,
 we assume that the set of solutions to \eqref{e:prob3} 
is non-empty.
\end{problem}
\begin{corollary}Let $(\SG_n)_{n\in\NN^{*}}$ be a $\HH$-valued random process such that, for
every $n\in\NN^*$, $\E[\| \SG_n\|^2]  < +\infty$,
let $w_1\colon\Omega\to \HH$ be a random variable such that $\E[\|w_1\|^2]<+\infty$, and set
\begin{equation}
\label{e:mainex1}
(\forall n\in\NN^{*})\quad
\begin{array}{l}
\operatorname{For}\;k=0,1,\ldots,\\
\left\lfloor
\begin{array}{l}
z_{n,k} = w_{n,k}- \gamma_n \scal{\mathfrak{B}_n}{e_k}\\
y_{n,k} = \prox_{\frac{\gamma_n}{1+\nu\gamma_n}\phi_k}\big(z_{n,k}/(1+\nu\gamma_n)\big)\\
w_{n+1,k} =  (1-\lambda_n)w_{n,k} + \lambda_ny_{n,k},\\
\end{array}
\right.\\[2mm]
\end{array}
\end{equation}
Suppose that conditions (A1), (A2), and  (A3)  are satisfied.
Then  the following hold.
\begin{enumerate}
\item Under the assumptions of Corollary \ref{cor:1},
 $\overline{w}$ is unique. In addition, setting, for every $n\in\NN^*$, $s_n=\E[\|w_n-\overline{w}\|^2]$,
 if $\theta\in]0,1[$, then \eqref{eq:Est1} holds, and if $\theta =1$  then \eqref{eq:Est11} holds. Moreover, 
\eqref{eq:Est111} holds.
\item Assume that $\smo$ is uniformly convex and (A4) is satisfied. Then $w_n \to \overline{w}$ a.s.
\item Assume that  $\smo$ is strictly convex and $\nabla \smo$ is weakly continuous, and (A4) is satisfied.
Then there exists a subsequence $(t_n)_{n\in\NN}$ such that $w_{t_n}\weakly \overline{w}$ a.s.
\end{enumerate} 
\end{corollary}
\begin{proof}
$G$ is $\nu$-strongly convex  by Parserval's identity, and 
 its proximity operator is computable \cite{livre1}. More precisely,  it follows from 
\cite[Proposition 23.29(i) and Proposition 23.34]{livre1} that the iteration \eqref{e:ex1} reduces 
to \eqref{e:mainex1}. Therefore, the statement follows from Corollaries
\ref{c:1mezzo} and  \ref{c:1}.
\end{proof}

\smallskip
\small{{\bf Acknowledgements}  This material is based upon work supported by the Center for Brains, Minds and Machines (CBMM), funded by NSF STC award CCF-1231216.
L. R. acknowledges the financial support of the Italian Ministry of Education, University and Research FIRB project RBFR12M3AC.
S. V. is member of the Gruppo Nazionale per
l'Analisi Matematica, la Probabilit\`a e le loro Applicazioni (GNAMPA)
of the Istituto Nazionale di Alta Matematica (INdAM).}
\def\cprime{$'$}



\end{document}